\documentclass[ssy,preprint]{imsart}

\usepackage{amsthm, amsmath, amssymb}
\usepackage{color,epsfig}
\usepackage{url}

\newtheorem{theorem}{Theorem}
\newtheorem{prop}[theorem]{Proposition}
\newtheorem{lemma}[theorem]{Lemma}
\newtheorem{coroll}[theorem]{Corollary}

\theoremstyle{remark}
\newtheorem{assumption}[theorem]{Assumption}
\newtheorem{remark}[theorem]{Remark}

\newcommand{\BR}{{\mathbb R}}

\newcommand{\BE}{{\mathbb E}}
\newcommand{\BP}{{\mathbb P}}

\newcommand{\CL}{{\mathcal L}}
\newcommand{\CS}{{\mathcal S}}
\newcommand{\CC}{{\mathcal C}}

\newcommand{\CJ}{{\mathcal J}}
\newcommand{\CI}{{\mathcal I}}
\newcommand{\CE}{{\mathcal E}}

\newcommand{\ov}{\overline}

\providecommand{\abs}[1]{\left\lvert#1\right\rvert}

\newcommand{\floor}[1]{\lfloor#1\rfloor}

\newcommand{\beql}[1]{\begin{equation}\label{#1}}
\newcommand{\eqn}[1]{(\ref{#1})}

\begin{document}

\begin{frontmatter}

\title{Tightness of invariant distributions of a large-scale flexible service system under a priority discipline}
\runtitle{Steady-state tightness, priority discipline}

\begin{aug}
  \author{\fnms{Alexander L.}  \snm{Stolyar}\ead[label=e1]{stolyar@research.bell-labs.com}}
  \and
  \author{\fnms{Elena}  \snm{Yudovina}\corref{}\thanksref{t1}\ead[label=e2]{yudovina@umich.edu}}

  \thankstext{t1}{The research of this author was supported by the NSF Graduate Research Fellowship.}

  \runauthor{A. Stolyar, E. Yudovina}

  \affiliation{Alcatel-Lucent Bell Labs and University of Michigan}

  \address{Murray Hill, NJ\\
          \printead{e1}}

  \address{Ann Arbor, MI\\ 
          \printead{e2}}

\end{aug}

\begin{abstract}
We consider 
large-scale service systems with multiple customer classes and multiple server pools; interarrival and service times are exponentially distributed, and mean service times depend both on the customer class and server pool. It is assumed that the allowed activities (routing choices) form a tree (in the graph with vertices being both customer classes and server pools). We study the behavior of the system under a {\em Leaf Activity Priority} (LAP) policy, which assigns static priorities to the activities in the order of sequential ``elimination" of the tree leaves.

We consider the scaling limit of the system as the arrival rate of customers and number of servers in each pool tend to infinity in proportion to a scaling parameter $r$, while the overall system load remains strictly subcritical. Indexing the systems by parameter $r$, we show that (a) the system under LAP discipline is stochastically stable for all sufficiently large $r$ and (b) the family of the invariant distributions  is tight on scales $r^{\frac12+\epsilon}$ for all $\epsilon > 0$. (More precisely, the sequence of invariant distributions, centered at the equilibrium point and scaled down by $r^{-(\frac12+\epsilon)}$, is tight.) 
\end{abstract}



\end{frontmatter}


\maketitle

\section{Introduction}

Large-scale service systems with heterogeneous customer and server populations bring up the need for efficient dynamic control policies that dynamically match arriving (or waiting) customers and available servers. It is desirable to have algorithms that avoid excessive customer waiting and do not rely on the knowledge of system parameters.

Consider a 
service system with multiple customer and server types, where the arrival rate of class $i$ customers is $\Lambda_i$, the service rate of a class $i$ customer by a type $j$ server is $\mu_{ij}$, and the server pool sizes are $B_j$. A desirable feature of a dynamic control is insensitivity to parameters $\Lambda_i$ and $\mu_{ij}$. That is, the assignment of customers to server pools should, to the maximal degree possible, depend only on the current system state (server occupancies, queue sizes), and not on prior knowledge of arrival rates or mean service times, because those parameters  may not be known in advance and, moreover, they may be changing in time.

If the system objective is to minimize the largest average load of any server pool, a ``static'' optimal control can be obtained by solving a linear program, called {\em static planning problem} (SPP), which has $B_j$'s, $\mu_{ij}$'s and $\Lambda_i$'s as parameters. An optimal solution to the SPP will prescribe optimal average rates $\Lambda_{ij}$ at which arriving customers should be routed to the server pools. Typically (in a certain sense) the solution to SPP is unique and the {\em basic activities}, i.e. routing choices $(ij)$ for which $\Lambda_{ij}>0$, form a tree; let us assume this is the case. Probabilistic routing with static probabilities $\Lambda_{ij}/\Lambda_{i}$ of routing a customer of type $i$ to server pool $j$, will balance the loads among different server pools and will avoid excessive customer waiting; however, in order to find the routing probabilities, it is necessary to know all of the parameters $\Lambda_i$, $B_j$, and $\mu_{ij}$ in advance. The \emph{Shadow Routing} policy in
\cite{Stolyar_Tezcan_underload} 
is a dynamic control policy, which achieves the load balancing objective without a priori knowledge of input rates $\Lambda_i$; in the process it ``automatically identifies'' the basic activity tree. Shadow Routing policy, however, does need to ``know'' the service rates $\mu_{ij}$.

In this paper we assume that the basic activity tree is known, but not the precise rates $\Lambda_i$, $\mu_{ij}$; we restrict the routing choices to activities within the basic activity tree. 
We consider the large-number-of-servers asymptotic regime, in which the arrival rate of customers and number of servers in each pool tend to infinity in proportion to a scaling parameter $r$; our focus is on the case where the overall system load remains strictly subcritical. 
In a previous paper \cite{SY10} we showed that a very natural load balancing policy considered e.g. by \cite{Gurvich_Whitt}, \cite{Armony_Ward}, \cite{Atar2009} may lead to instability at the system equilibrium point: in particular,
for certain parameter settings
 \cite[Theorem 7.2]{SY10} 
demonstrated the non-tightness (in fact -- evanescence to infinity) of invariant measures on the diffusion, $r^{\frac12}$- scale.
(More precisely, this means that  the sequence of invariant distributions, 
centered at the equilibrium point and scaled down by $r^{-\frac12}$, is non-tight, and moreover -- escapes to infinity.)

In this paper we consider a different algorithm, which we call the {\em Leaf Activity Priority} (LAP) policy. As specified above, no precise knowledge of the rates $\Lambda_i$ and $\mu_{ij}$ is required, besides the knowledge of the basic activity tree,
and routing is restricted to basic activities only. The policy
assigns static priorities to the activities in the order of sequential ``elimination" of the tree leaves.
The precise definition will be given in Section \ref{section:LAP}. Assuming strictly subcritical load,
for this policy we first prove that the system is stochastically stable for all sufficiently large values of $r$.
(In contrast to load balancing policies, the stability under LAP is not ``automatic".)
Next, we demonstrate the $r$-scale (fluid-scale) tightness of stationary distributions;
this fact is closely related to stability -- both are``consequences" of the relatively ``benign" behavior of 
the system on the fluid scale. Then, we obtain a much stronger tightness result, namely that the invariant distributions are tight on the $r^{\frac12 + \epsilon}$-scale, for any $\epsilon>0$; this is the main contribution of the paper, which involves
the analysis of the process under hydrodynamic and local-fluid scaling (in addition to ``standard" fluid scaling).
We believe that our analysis can be extended to prove still stronger, diffusion scale ($r^{1/2}$) tightness; this is work in progress.

For a general review of literature on the large-number-of-servers asymptotic regime, including design and analysis of efficient control algorithms,
see e.g.  \cite{Gurvich_Whitt,Stolyar_Tezcan_underload} and references therein.

The rest of the paper is organized as follows. In Section~\ref{section:model} the model, the asymptotic regime, LAP discipline
and basic notation and introduced. The main results are stated in Theorem~\ref{th-stationary-scale-main} of Section~\ref{sec-main-res},
with its statements (i) and (ii) being the stability and tightness results, respectively. Section~\ref{sec-fluid-scaling} contains the analysis of the process on the fluid scale, which leads to establishing stability (Theorem~\ref{th-stationary-scale-main}(i)) 
and fluid scale ($r$-scale) tightness of stationary distributions. In Section~\ref{sec-main-proof}, using the fluid-scale tightness as a starting point, we prove the $r^{1/2+\epsilon}$-scale tightness 
(Theorem~\ref{th-stationary-scale-main}(ii)); this is the key part of the paper, which involves the analysis of system dynamics under LAP discipline under hydrodynamic and local-fluid scaling.

\section{Model}
\label{section:model}

\subsection{The model; Static Planning (LP) Problem}
Consider the model in which there are $I$ customer classes, labeled $1,2,\dotsc,I$, and $J$ server pools, labeled $1,2,\dotsc,J$. 
(Servers within pool $j$ are referred to as class $j$ servers. Also,
throughout this paper the terms ``class'' and ``type" are used interchangeably.)
The sets of customer classes and server pools will be denoted by $\CI$ and $\CJ$, respectively. We will use the indices $i$, $i'$ to refer to customer classes, and $j$, $j'$ to refer to server pools.

We are interested in the scaling properties of the system as it grows large. The meaning of ``grows large" is as follows. We consider a sequence of systems indexed by a scaling parameter $r$. As $r$ grows, the arrival rates and the sizes of the service pools, but not the speed of service, increase. Specifically, in the $r$th system, customers of type $i$ enter the system as a Poisson process of rate $\lambda^r_i = r \lambda_i$, while the $j$th server pool has $r \beta_j$ individual servers. (All $\lambda_i$ and $\beta_j$ are positive parameters.) Customers may be accepted for service immediately upon arrival, or enter a queue; there is a separate queue for each customer type. Customers do not abandon the system. When a customer of type $i$ is accepted for service by a server in pool $j$, the service time is exponential of rate $\mu_{ij}$; the service rate depends both on the customer type and the server type, but \emph{not} on the scaling parameter $r$. If customers of type $i$ cannot be served by servers of class $j$, the service rate is $\mu_{ij} = 0$.

\begin{remark}
Strictly speaking, the quantity $\beta_j r$ may not be an integer, so we should define the number of servers in pool $j$ as, say, $\floor{\beta_j r}$. However, the change is not substantial, and will only unnecessarily complicate the notation.
\end{remark}

Consider the following, load-balancing, {\em static planning problem} (SPP):
\begin{subequations}\label{eqn:Static LP}
\begin{equation}
\min_{\lambda_{ij}, \rho} \rho,
\end{equation}
subject to
\beql{lp-constraint0}
\lambda_{ij} \geq 0, ~~\forall i,j
\end{equation}
\beql{lp-constraint1}
\sum_j \lambda_{ij} = \lambda_i , ~~\forall i
\end{equation}
\beql{lp-constraint2}
\sum_i \lambda_{ij} / (\beta_j \mu_{ij}) \leq \rho , ~~\forall j.
\end{equation}
\end{subequations}

Throughout this paper we will always make the following two assumptions about the solution to the SPP \eqref{eqn:Static LP}:
\begin{assumption}[Complete resource pooling]
\label{assuption: CRP}
The SPP \eqref{eqn:Static LP} has a unique optimal solution $\{\lambda_{ij}, ~i\in \CI, ~j\in \CJ\}, \rho$. Define the \emph{basic activities} to be the pairs, or edges, $(ij)$ for which $\lambda_{ij} > 0$. Let $\CE$ be the set of basic activities. We further assume that the unique optimal solution is such that $\CE$ forms a  tree in the (undirected) graph with vertices set $\CI \cup \CJ$.
\end{assumption}

\begin{assumption}[Underload]
\label{CRP:2} The optimal solution to \eqref{eqn:Static LP} has $\rho < 1$.
\end{assumption}

\begin{remark}
Assumption~\ref{assuption: CRP} is the \emph{complete resource pooling} (CRP) condition, which holds ``generically'' in a certain sense; see \cite[Theorem 2.2]{Stolyar_Tezcan}. Assumption~\ref{CRP:2} is essential for the main results of the paper ($r^{\frac12+\epsilon}$-scale tightness),
but many of the auxiliary results hold (along with their proofs) for the critically loaded case $\rho=1$.
\end{remark}

Note that under the CRP condition, all  (``server pool capacity'')  constraints \eqn{lp-constraint2} are binding:  $\sum_i \lambda_{ij}/(\beta_j \mu_{ij}) = \rho, ~\forall j$.
This in particular means that the optimal solution to SPP is such that, if a system with parameter $r$ will route type $i$ customers to pool $j$ at the rate $\lambda_{ij}r$, the server pool average loads will be minimized and ``perfectly balanced''.

In this paper, we assume that the basic activity tree is known in advance, and restrict our attention to the basic activities only. Namely, we assume that a type $i$ customer service in pool $j$ is allowed only if $(ij)\in\CE$. (Equivalently, we can a priori assume that $\CE$ is the set of {\em all} possible activities, i.e. $\mu_{ij} = 0$ when $(ij)\not\in\CE$, and $\CE$ is a tree. In this case CRP requires that all feasible activities are basic.) 
For a customer type $i$, let $\CS(i) = \{j: (ij)\in\CE \}$; for a server type $j$, let $\CC(j) = \{i: (ij)\in\CE \}$.

Under the CRP condition, optimal dual variables $\nu_i, ~i\in \CI$, and $\alpha_j, ~j\in \CJ$, corresponding to constraints \eqn{lp-constraint1} and \eqn{lp-constraint2}, respectively, are unique and all strictly positive. The dual variable $\nu_i$ is interpreted as the ``workload'' associated with one type $i$ customer, and $\frac{\alpha_j}{\beta_j}$ is interpreted as the (average) rate at which one server in pool $j$ processes workload when it is busy, regardless of the customer type on which it is working, as long as $i\in \CC(j)$. The dual variables satisfy the relations $\nu_i \mu_{ij} = \alpha_j/\beta_j$ for any $(ij)\in\CE$, and $\sum_j \alpha_j=1$, which in particular imply that 
\beql{eqn:workload}
\sum_i \nu_i \lambda_i = \sum_i \sum_j \nu_i \lambda_{ij}
=\sum_i \sum_j\lambda_{ij} \frac{\alpha_j}{\beta_j \mu_{ij}}=
\sum_j \alpha_j\sum_i \frac{\lambda_{ij}}{\beta_j \mu_{ij}}=
\rho.
\end{equation}
Given $\rho<1$, this means, for example, that when all servers in the system are busy, the total rate $\sum_i \nu_i \lambda_i r$ at which new workload arrives in the system is strictly less than the rate $\sum_j \alpha_j r =r$ at which it is served.

\begin{remark}
Although \eqref{eqn:Static LP} is the load-balancing SPP, and the notions introduced in this subsection are defined in terms of this SPP,
the policy we consider in this paper (defined in Section \ref{section:LAP}) is {\em not} a load balancing policy.
In particular, the system equilibrium point under the policy, will {\em not} balance server pool loads, but rather will keep all pools, except one,
fully occupied.
\end{remark}

\subsection{Leaf activity priority (LAP) policy}\label{section:LAP}

For the rest of the paper, we analyze the performance of the following policy, which we call {\em leaf activity priority} (LAP).
The first step in its definition is the assignment of priorities to customer classes and activities.

Consider the basic activity tree, and assign priorities to the edges as follows. First, we assign priorities to customer classes by iterating the following procedure:
\begin{enumerate}
\item Pick a leaf of the tree;
\item If it is a customer class (rather than a server class), assign to it the highest priority that hasn't yet been assigned;
\item Remove the leaf from the tree.
\end{enumerate}
Without loss of generality, we assume the customer classes are numbered in order of priority (with 1 being highest). We now assign priorities to the edges of the basic activity tree by iterating the following procedure:
\begin{enumerate}
\item Pick the highest-priority customer class;
\item If this customer class {\em is} a leaf, pick the edge going out of it,
assign this edge the highest priority that hasn't yet been assigned, and remove the edge together with the customer class;
\item If this customer class is {\em not} a leaf, then pick any edge from it to a server class leaf (such necessarily exists),
assign to this edge the highest priority that hasn't yet been assigned, and remove the edge.
\end{enumerate}
It is not hard to verify that this algorithm will successfully assign priorities to all edges; it suffices to check that at any time the highest remaining priority
customer class will have at most one outgoing edge
to a non-leaf server class.

\begin{remark}\label{rem: priorities}
This algorithm does {\em not} produce a unique assignment of priorities, neither for the customer classes nor for the activities,
because there may be multiple options for picking a next leaf or edge to remove, in the corresponding procedures.
This is not a problem, because our results hold for {\em any} such assignment. 
Different priority assignments  may correspond to different equilibrium points (defined below in Section~\ref{sec: LAP equilibrium}); 
once we have picked a particular priority assignment, there is a (unique) corresponding equilibrium point, 
and we will be showing steady-state tightness around that point. 
Furthermore, the flexibility in assigning priorities may be a useful feature in practice. For example, it is easy to specialize the above
priority assignment procedure so that the lowest priority is given to any a priori picked activity. 
\end{remark}

We illustrate one such priority assignment in Figure~\ref{fig:priorities}.
\begin{figure}[htbp]
\includegraphics{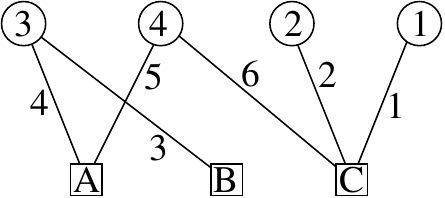}
\caption{An example assignment of priorities to customer classes and activities to an example network. Circles represent customer classes, squares represent server pools.}
\label{fig:priorities}
\end{figure}

We will write $(ij) < (i'j')$ to mean that activity $(ij)$ has higher priority than activity $(i'j')$. 
It follows from the priority assignment algorithm that 
$i < i'$ (customer class $i$ has higher priority than $i'$) implies $(ij) < (i'j')$. In particular,  if $j=j'$, we have $(ij) < (i'j)$ if and only if $i < i'$. 
Without loss of generality, we shall assume that the server classes are numbered so that the lowest-priority activity is $(IJ)$. (In Figure~\ref{fig:priorities}, this corresponds to assigning the number 3 to server pool $C$.)

Now we define the LAP policy itself. It consists of two parts: routing and scheduling. ``Routing'' determines where an arriving customer goes if it sees available servers of several different types. ``Scheduling'' determines which waiting customer a server picks if it sees customers of several different types waiting in queue.

{\bf Routing:} An arriving customer of type $i$ picks an unoccupied server in the pool $j \in \CS(i)$ such that $(ij) \leq (ij')$ for all $j' \in \CS(i)$ with idle servers. If no server pools in $\CS(i)$ have idle servers, the customer queues.

{\bf Scheduling:} A server of type $j$ upon completing a service picks the customer from the queue of type $i \in \CC(j)$ such that $i \leq i'$ for all $i' \in \CS(i)$ with $Q_{i'} > 0$. If no customer types in $\CC(j)$ have queues, the server remains idle.

We introduce the following notation (for the system with scaling parameter $r$):\\
$\Psi^r_{ij}(t)$, the number of servers of type $j$ serving customers of type $i$ at time $t$;\\
$Q^r_i(t)$, the number of customers of type $i$ waiting for service at time $t$.

\subsection{LAP equilibrium point}\label{sec: LAP equilibrium}
Informally speaking, the equilibrium point $(\psi^*_{ij},q^*_i)_{(ij) \in \CE, i \in \CI}$ is the desired operating point for the (fluid scaled)
vector $(\Psi^r_{ij}/r,Q^r_i/r)_{(ij) \in \CE, i \in \CI}$ of occupancies and queue lengths under the LAP policy. Specifically, we will be showing that in steady state the fluid-scaled vector converges in distribution to the equilibrium point, and will then show that the deviations from it are small. We define the equilibrium point below; it will be the stationary point of the fluid models defined in Section~\ref{sec-fluid-scaling}.

The LAP discipline is not designed with load balancing in mind, so its equilibrium point does {\em not}, of course, achieve load balancing
among the server pools. To define it, we recursively define the quantities $\lambda_{ij}\ge 0$, which have the meaning of routing rates, scaled down by factor $1/r$. (These $\lambda_{ij}$ are {\em not} the same as those given by the optimal solution to the SPP \eqn{eqn:Static LP}.) For the activity $(1j)$ with the highest priority, define either $\lambda_{1j} = \lambda_1$ and $\psi^*_{1j} = \frac{\lambda_1}{\mu_{1j}}$, or $\psi^*_{1j} = \beta_j$ and $\lambda_{1j} = \beta_j \mu_{1j}$, according to whichever is smaller. Replace $\lambda_1$ by $\lambda_1 - \lambda_{1j}$ and $\beta_j$ by $\beta_j - \psi_{1j}^*$, and remove the edge $(1j)$ from the tree. We now proceed similarly with the remaining activities.

Formally, set
\[
\lambda_{ij} = \min\left(\lambda_i - \sum_{j': (ij') < (ij)} \lambda_{ij'}, \mu_{ij}\left(\beta_j - \sum_{i'<i} \frac{\lambda_{i'j}}{\mu_{i'j}}\right)\right).
\]
Since the definition is in terms of higher-priority activities, this defines the $(\lambda_{ij})_{(ij) \in \CE}$ uniquely. The LAP equilibrium point is defined to be the vector
\[
(\psi^*_{ij},q^*_i)_{(ij) \in \CE, i \in \CI}
\]
given by
\beql{eqn:LAP equilibrium}
\psi^*_{ij} = \frac{\lambda_{ij}}{\mu_{ij}}, \quad q^*_i = 0 \text{ for all $(ij) \in \CE$, $i \in \CI$}.
\end{equation}
(Since we're in the underloaded case $\rho<1$, all queues should be 0 at equilibrium.)

To avoid trivial complications,  throughout the paper we make the following assumption:
\begin{assumption}\label{ass:rho}
If $(\psi_{ij})_{(ij) \in \CE}$ are such that $\psi_{ij} \geq 0$, $\lambda_i = \sum_j \mu_{ij} \psi_{ij}$, and $\sum_i \psi_{ij} \leq \beta_j$ for all $j$, then 
$\psi_{ij} > 0$ for all $(ij)\in \CE$.
In particular, the equilibrium point satisfies this condition and, moreover, it is such that $\sum_i \psi^*_{ij} = \beta_j$ for all $j < J$ and $\sum_i \psi^*_{iJ} < \beta_J$.
\end{assumption}
The assumption means that the system needs to employ (on average) all activities in order to be able to handle the load.
It holds, for example, whenever $\rho$ is sufficiently close to 1.  Indeed, suppose the arrival rates $(\lambda_i)_{i\in\CI}$ are such that $\rho = 1-\epsilon$, and consider the system with arrival rates $\hat\lambda_i = \frac{1}{1-\epsilon} \lambda_i$. The basic activity tree $\CE$ will be the same for $(\hat \lambda_i)_{i\in\CI}$ as for $(\lambda_i)_{i\in\CI}$. 
Since $\rho = 1$ for $(\hat \lambda_i)_{i\in\CI}$ and CRP holds, there exists a unique set of $(\hat\psi_{ij})_{(ij) \in \CE}$ that satisfies the conditions, and it has $\hat\psi_{ij} > 0$ for all $(ij) \in \CE$. Also, if $\epsilon$ is sufficiently small, we must have $\psi_{ij} \approx \hat\psi_{ij}$ for all $(ij) \in \CE$, and hence $\psi_{ij} > 0$ for all $(ij) \in \CE$.

\begin{remark}
Assumption~\ref{ass:rho} is technical -- our main result, Theorem~\ref{th-stationary-scale-main}, can be proved without it, by following the approach presented in the paper.
But, it simplifies the statements and proofs 
of many auxiliary results, and thus substantially improves the exposition.
\end{remark}

\begin{remark}
Although the LAP equilibrium point does not achieve load balancing, when system is heavily loaded  (i.e. $\rho$ is close to $1$),
the load balancing is achieved approximately, in the sense that all queues are kept small and all servers are almost fully loaded --
which is the best any ``load balancer'' could do (when $\rho$ is close to $1$).
\end{remark}

\subsection{Basic notation}

Vector $(\xi_i, ~i\in \CI)$, where $\xi$ can be any symbol, is often written as $(\xi_i)$;
similarly, 
$(\xi_j, ~j\in \CJ)=(\xi_j)$
 and 
$(\xi_{ij}, ~(ij)\in \CE)=(\xi_{ij})$. 
Furthermore, we often use notation $(\xi_i, \eta_{ij})$   to mean  $((\xi_i, ~i\in \CI), (\eta_{ij}, ~(ij)\in \CE))$,
and similar notations as well.
Unless specified otherwise, $\sum_i \xi_{ij} = \sum_{i \in \CC(j)} \xi_{ij}$ and $\sum_j \xi_{ij} = \sum_{j \in \CS(i)} \xi_{ij}$. For functions (or random processes) $(\xi(t), ~t\ge 0)$ we often write $\xi(\cdot)$. (And similarly for functions with domain different from $[0,\infty)$.) So, for example, $(\xi_i(\cdot))$ signifies $((\xi_i(t), i\in \CI), ~t\ge 0)$. 

The symbol $\implies$ denotes convergence in distribution of random variables in the Euclidean space $\mathbb{R}^d$ (with appropriate dimension $d$).
The symbol $\to$ denotes ordinary convergence in $\BR^d$. Standard Euclidean norm of a vector $x\in \mathbb{R}^d$ is denoted $\abs{x}$, while 
$\|x\|$ denotes the $L_1$-norm (sum of absolute values of the components); 
{\em u.o.c.} means {\em uniform on compact sets} convergence of functions, with the domain defined explicitly or
by the context.
For $x \in \BR$, $\floor{x}$ is the greatest integer less than or equal to $x$.

\section{Main result}
\label{sec-main-res}

We are now in position to state our main result.

\begin{theorem}\label{th-stationary-scale-main}
Consider the sequence of systems under LAP policy,
 in the scaling regime and under the assumptions
specified in Section~\ref {section:model}, with $\rho<1$. Then:\\ 
(i) For all sufficiently large $r$, the system is stable, i.e. the countable Markov 
chain
$(\Psi^r_{ij}(\cdot), Q^r_i(\cdot))$ is positive recurrent.\\
(ii) For any $\epsilon>0$, the stationary distribution
of $r^{-1/2-\epsilon} (\Psi^r_{ij}(\cdot)-\psi_{ij}^* r, Q^r_i(\cdot))$ weakly converges to $0$.
\end{theorem}

The proof is given in the rest of the paper, and consists roughly of two stages.
First, we study the process under the fluid 
scaling $r^{-1} (\Psi^r_{ij}(\cdot), Q^r_i(\cdot))$, which allows us to prove stability
 and statement (ii) for $\epsilon=1/2$. Then we need a more detailed analysis, involving
{\em hydrodynamic}  and {\em local-fluid} scaling of the process, to prove (ii) for any 
$\epsilon>0$.

Throughout the paper, we will use the following additional notation for the system variables.
For a system with parameter $r$, we denote:\\
$X^r_i(t)=\sum_j\Psi^r_{ij}(t)+Q^r_i(t)$ is the total number of customers of type $i$ in the system at time $t$;\\
$A^r_i(t)$ is the total number of customers of type $i$ exogenous arrivals into the system in interval $[0,t]$;\\
$D^r_{ij}(t)$ is the total number of customers of type $i$ that completed the service in pool $j$ (and departed the system) in interval $[0,t]$;\\
$\Xi^r_{ij}(t)$ is the total number of customers of type $i$ that entered service in pool $j$ in interval $[0,t]$.\\
There are some obvious relations between realizations of these processes: $Q^r_i(t) = Q^r_i(0) + A^r_i(t) - \sum_j \Xi^r_{ij}(t)$; $Q^r_i(t) >0$
implies $\sum_{i'} \Psi^r_{i'j}(t) = \beta_j r$ for each $j\in \CS(i)$; and so on.

We can and will assume that a random realization of the system with parameter $r$ is determined by its initial state and realizations of ``driving" unit-rate, mutually independent, Poisson process $\Pi_i^{(a)}(\cdot), i\in \CI$ and 
$\Pi_{ij}^{(s)}(\cdot), (ij)\in \CE$, as follows:
$$
A_i^r(t) = \Pi_i^{(a)}(\lambda_i r t), ~~~~D_{ij}^r(t)=\Pi_{ij}^{(s)}\Bigl(\mu_{ij} \int_0^t \Psi^r_{ij}(u) du \Bigr);
$$
the driving Poisson processes are common for all $r$.

\section{Fluid scaling}
\label{sec-fluid-scaling}
We begin by analyzing the LAP discipline on the fluid scale. Namely, consider the scaling
\begin{align*}
\Bigl(\psi^r_{ij}(t),q^r_i(t), x^r_i(t), a^r_i(t), d^r_{ij}(t),\xi^r_{ij}(t)\Bigr)\\
= \frac1r\Bigl (\Psi^r_{ij}(t),Q^r_i(t), X^r_i(t), A^r_i(t), D^r_{ij}(t),\Xi^r_{ij}(t)\Bigr).
\end{align*}

\begin{prop}\label{th-local-fluid-convergence}
Suppose $(\psi^r_{ij}(0),q^r_i(0)) \to (\psi_{ij}(0), q_i(0))$. Then, w.p.1,  for any subsequence $r \to \infty$ there exists a further subsequence along which $(\psi^r_{ij}(\cdot),q^r_i(\cdot), x^r_i(\cdot), a^r_i(\cdot), d^r_{ij}(\cdot),\xi^r_{ij}(\cdot))$ converges uniformly on compact sets to a set $(\psi_{ij}(\cdot), q_i(\cdot), x_i(\cdot), a_i(\cdot), d_{ij}(\cdot),\xi_{ij}(\cdot))$ of Lipschitz continuous functions satisfying conditions \eqref{eqn:fluid model}. The conditions involving derivatives are to be satisfied 
at all regular points of the limiting set of functions. (A time point $t\ge 0$ 
is {\em regular} if both minimum and maximum of any subset of component functions
have derivatives at $t$. All points $t\ge 0$ are regular, except a subset of zero
 Lebesgue measure.)
\end{prop}

The fluid model conditions are: 
\begin{subequations}\label{eqn:fluid model}
\begin{equation}\label{eqn-first}
q_i(t) \geq 0, ~~ \forall i \in \CI; \quad \psi_{ij}(t) \geq 0, ~~ \forall (ij) \in \CE; \quad \sum_i \psi_{ij}(t) \leq \beta_j, ~~ \forall j \in \CJ
\end{equation}
\begin{equation}
a_i(t) = \lambda_i t, ~~ \forall i \in \CI; \qquad d_{ij}(t) = \int_0^t\mu_{ij}\psi_{ij}(s) ds, ~~ \forall (ij) \in \CE
\end{equation}
\begin{align}\label{eqn-q}
&q_i(t) = q_i(0) +  a_i(t) - \sum_j \xi_{ij}(t),~~ \forall i \in \CI;\\
&\psi_{ij}(t) =\psi_{ij}(0) +  \xi_{ij}(t) - d_{ij}(t),~~ \forall (ij) \in \CE\notag
\end{align}
\begin{equation}\label{eqn-x}
x_i(t) = q_i(t) + \sum_j \psi_{ij}(t) = x_i(0) + \lambda_i t - \sum_j \int_0^t \mu_{ij}\psi_{ij}(s) ds,~~ \forall i \in \CI
\end{equation}
\begin{equation}\label{eqn-key1}
\sum_i \psi_{ij}(t) = \beta_j, ~~ \text{whenever $q_{i'}(t)>0$ for at least one $i' \in \CC(j)$}
\end{equation}
\begin{equation}\label{eqn-key2}
\frac{d}{dt}\xi_{ij}(t) = 0 , ~~ \text{whenever $q_{i'}(t)>0$ for at least one $i' \in \CC(j)$, $i'<i$}
\end{equation}
\begin{equation}\label{eqn-key3}
\frac{d}{dt}\xi_{ij}(t) = 0 , ~~ \text{whenever $\sum_{k} \psi_{kj'}(t) < \beta_{j'}$ for at least one $(ij') < (ij)$}
\end{equation}
\begin{multline}\label{eqn:last-2}
\frac{d}{dt} \xi_{ij}(t) = \sum_{i'} \mu_{i'j} \psi_{i'j}(t) -  \sum_{(i'j) < (ij)} \frac{d}{dt} \xi_{ij'}(t),\\
\text{whenever $q_i(t)>0$ (and then necessarily $\sum_{k} \psi_{kj}(t) = \beta_j$)}
\end{multline}
\begin{multline}\label{eqn:last-1}
\frac{d}{dt} \xi_{ij}(t) = \lambda_{i} - \sum_{(ij') < (ij)} \frac{d}{dt} \xi_{ij'}(t),\\
\text{whenever $\sum_{k} \psi_{kj}(t) < \beta_j$ (and then necessarily $q_i(t)=0$)}
\end{multline}
\begin{multline}\label{eqn:last}
\frac{d}{dt} \xi_{ij}(t) = \min \left(\lambda_{i} - \sum_{(ij') < (ij)} \frac{d}{dt} \xi_{ij'}(t), \right.\\
\left.\sum_{i'} \mu_{i'j} \psi_{i'j}(t) -  \sum_{(i'j) < (ij)} \frac{d}{dt} \xi_{ij'}(t) \right)\\
\text{whenever $q_i(t)=0$ and $\sum_{k} \psi_{kj}(t) = \beta_j$.}
\end{multline}
\end{subequations}

\begin{proof}
[Proof of Proposition~\ref{th-local-fluid-convergence}]
The proof of convergence fact and
of the basic conditions
\eqn{eqn-first}-\eqn{eqn-x} of the limit,
is very standard. Indeed,
it follows from the Functional Strong Law of Large Numbers (FSLLN) for the driving processes,
and the scaling applied, that w.p.1 each component function is asymptotically
Lipschitz. For example, for each scaled departure process we have: 
w.p.1, for a fixed large $C>0$ and any $0\le t_1 < t_2 <\infty$, 
$$
\limsup_{r\to\infty} d^r_{ij}(t_2)-d^r_{ij}(t_1) < C(t_2-t_1).
$$
This implies that, w.p.1 any subsequence of $r$ has a further
subsequence along which a u.o.c. convergence $d^r_{ij}(\cdot)\to d_{ij}(\cdot)$ holds,
where $d_{ij}(\cdot)$ is Lipschitz. Similar convergence property 
holds for each arrival process. 
From here we obtain the convergence (along a subsequence) for all other components.
Then, relations \eqn{eqn-first}-\eqn{eqn-x} are inherited from the corresponding
conservation laws for the pre-limit trajectories.

Properties \eqn{eqn-key1}-\eqn{eqn:last-1} easily follow from the priority rule
of LAP;
it suffices to consider the behavior of pre-limit trajectories in a
small time interval $[t,t+\delta]$ when $r$ sufficiently large.
(See e.g. \cite[Theorem 1]{Mandelbaum_Stolyar} for this type of argument.)

Finally, to show \eqn{eqn:last} we recall that, by the priority assignment procedure,
for the activity $(ij)$:
either $(ij)$ has the lowest priority among activities associated with customer class $i$
or $(ij)$ has the lowest priority among activities associated with server pool $j$ (or both).
Taking into account that point $t$ is regular (which in particular implies
$q'_i(t)=0$ and $(d/dt) \sum_k \psi_{kj}(t) = 0$), we easily see that
in the former case the only possibility is that
$$
\frac{d}{dt} \xi_{ij}(t) = \lambda_{i} - \sum_{(ij') < (ij)} \frac{d}{dt} \xi_{ij'}(t) \le
\sum_{i'} \mu_{i'j} \psi_{i'j}(t) -  \sum_{(i'j) < (ij)} \frac{d}{dt} \xi_{ij'}(t),
$$
and in the latter case we must have
$$
\frac{d}{dt} \xi_{ij}(t) = 
\sum_{i'} \mu_{i'j} \psi_{i'j}(t) -  \sum_{(i'j) < (ij)} \frac{d}{dt} \xi_{ij'}(t)
\le \lambda_{i} - \sum_{(ij') < (ij)} \frac{d}{dt} \xi_{ij'}(t).
$$
This implies \eqn{eqn:last}. We omit further details of the proof which are, again, 
rather standard.
\end{proof}

We call any Lipschitz solution $(\psi_{ij}(\cdot), q_i(\cdot), x_i(\cdot), a_i(\cdot), d_{ij}(\cdot),\xi_{ij}(\cdot))$
of \eqref{eqn:fluid model}  a \emph{fluid model} of the system with initial state $(\psi_{ij}(0), q_i(0))$; a set $(\psi_{ij}(\cdot), q_i(\cdot))$,
which is a projection of a fluid model we often call a fluid model as well.

\begin{remark}\label{rem:uniqueness}
It will not be important for the results in the paper whether the fluid model with given initial conditions is unique; all that will matter is the long-term behavior of all fluid models with given initial conditions. 
\end{remark}

\begin{prop}\label{th-fluid-drains}
For any $\epsilon' > 0$ and any $K>0$ there exists a finite time $T=T(K)$ such that all fluid models whose starting state satisfies $\abs{(\psi_{ij}(0), q_i(0))} \le K$ have $\sum_i \psi_{ij}(t) = \beta_j,~\forall j < J$, $q_i(t) = 0,~\forall i \in \CI$, and  $\abs{\psi_{ij}(t) - \psi^*_{ij}}<\epsilon'$ for all $(ij) \in \CE$, for all $t \geq T(K)$.
\end{prop}

\begin{proof}[Sketch of proof]
For the highest priority activity $(1j)$ there are two cases.\\
Case a: Type $1$ is a leaf. 
In this case $j$ is the unique server to which type 1 jobs are allowed to go, and they have the highest priority there.
Pick a small $\delta>0$. After a finite time (uniformly bounded above, across
all starting states as in the proposition statement),
the condition $\psi_{1j}(t)\ge \psi_{1j}^* - \delta$ must hold, because
$\psi_{1j}(t)< \psi_{1j}^* - \delta$ implies that $(d/dt)\psi_{1j}(t)$ is positive 
and bounded away from $0$. After such time, $q_1(t)>0$ implies 
$\sum_i' \psi_{i'j}(t)=\beta_j$ and (recall that $\delta$ is small) 
 $\lambda_1\le \mu_{1j}\psi_{1j}(t) -\delta_1$ for some $\delta_1>0$; and therefore
$(d/dt)q_1(t)\le -\delta_1$. We conclude that after a finite time (uniformly bounded above)
we must have $q_1(t)=0$. This in turn implies that $(d/dt)\psi_{1j}(t)$ is negative
and bounded away from $0$ as long as $\psi_{1j}(t)> \psi_{1j}^* + \delta$.
Thus, $|\psi_{1j}(t)-\psi_{1j}^*|\le \delta$ and $q_1(t)=0$ after a bounded time.\\
Case b: Pool $j$ is a leaf. Then Assumption~\ref{ass:rho} implies $\psi_{1j}^*=\beta_j$
and $\lambda_1>\mu_{1j} \beta_j$. 
In this case, $\psi_{1j}(t)=\psi^*_{1j}$ starting at some time (that is bounded uniformly on initial states), simply because $(d/dt)\psi_{1j}(t) \ge \lambda_1 - \mu_{1j} \beta_j >0$
as long as $\psi_{1j}(t)<\beta_j$.\\
We see that, in either case a or b, for arbitrarily small $\delta>0$, there exists $T_1=T_1(\delta)$ such that $|\psi_{1j}(t)-\psi_{1j}^*|<\delta$.\\
We proceed by induction on the activity priorities and, using Assumption~\ref{ass:rho}, easily establish analogous properties for every activity $(i'j')$.
This implies the result. We omit details.
\end{proof}

\begin{theorem}
\label{thm-fluid-stability}
For all sufficiently large $r$, the LAP discipline stabilizes the network (in the sense of positive recurrence of the underlying Markov process). Moreover, the sequence of invariant distributions of $(\psi^r_{ij}, q^r_i)$ is tight, and the invariant distributions converge weakly to the point mass at the equilibrium point.
\end{theorem}

Before we proceed with the proof, we need the following lemma.
\begin{lemma}\label{lem-for-lyapunov}
There exists $T_1 > 0$ such that for any $T_2 > T_1$ there exists a sufficiently large $C=C(T_2)$ for which the following holds. For any $\epsilon > 0$, 
\[
\BP \left\{ \bigl\lvert \sum_{(ij)} \nu_i (d^r_{ij}(T_2)-d^r_{ij}(T_1)) - (T_2-T_1) \bigr\rvert \geq \epsilon \right\} \to 0,
\]
as $r\to\infty$, uniformly on initial states with $\max_{i \in \CI} q^r_i(0) \geq C$.
\end{lemma}

In turn, to prove this lemma, we will need to use fluid models with infinite initial states.
Note that we cannot appeal directly to the properties of ``standard" fluid models defined earlier, because we require convergence that is uniform in all large initial states.
So, we need the following version of a fluid limit result. 
We will use notation $\bar \BR = \BR\cup\{\infty\}$ for the the one-point 
compactification of $\BR$.

\begin{prop}\label{th-fluid-infinite-init}
Consider a sequence of fluid-scaled processes \linebreak
$(\psi^r_{ij}(\cdot), q^r_i(\cdot))$
with deterministic initial states such that 
$\abs{(\psi^r_{ij}(0), q^r_i(0))} = C'(r) \to \infty$ and
$$
(\psi^r_{ij}(0), q^r_i(0)) \to (\psi_{ij}(0), q_i(0)),
$$
where each $q^r_i(0)$ and $q_i(0)$ is viewed as an element of $\bar \BR$.
Partition the customer classes as $\CI = \CI^\infty \cup \CI^0$, 
where $q_i(0) = \infty$ for $i \in \CI^\infty$, and $q_i(0) < \infty$ for $i \in \CI^0$.
(Necessarily, $\CI^\infty$ is non-empty.)
Then, with probability 1, any subsequence of trajectories has a further subsequence which converges u.o.c. to a fluid model, satisfying same conditions as \eqn{eqn:fluid model}, except that
for all $i \in \CI^\infty$ the queue length $q_i(t)=\infty, \forall t\ge 0$. 
Moreover, all such fluid models are such that, uniformly on all of them, starting at some finite time $T'_1$, all server pools are fully occupied: $\sum_i \psi_{ij}(t)=\beta_j$, 
$t\ge T'_1$, $\forall j$.
\end{prop}

Proof of this result is very similar to that of Proposition~\ref{th-fluid-drains}
(and in fact simpler), so it is not spelled out here. 
We just note that Assumption~\ref{ass:rho} is essential in showing that {\em all}
server pools are occupied after a finite time. Without the assumption, we could still
show that the occupancy becomes {\em strictly greater} than at the equilibrium point,
and that would be enough for our purposes; however, it would make 
Proposition~\ref{th-fluid-infinite-init} statement and proof more cumbersome.

\begin{proof}[Proof of Lemma~\ref{lem-for-lyapunov}]
Let us choose $T_1=2 T'_1$, where $T'_1$ is as in Proposition~\ref{th-fluid-infinite-init}.
 Now, if the lemma statement would not hold, then for some fixed $\epsilon'>0$ we could find a sequence
of systems with $\abs{(\psi^r_{ij}(0), q^r_i(0))} = C'(r) \to \infty$, such that 
\[
\limsup_r \BP \left\{ \bigl\lvert \sum_{(ij)} \nu_i (d^r_{ij}(T_2)-d^r_{ij}(T_1))   
- (T_2-T_1) \bigr\rvert \geq \epsilon' \right\} > 0.
\]
This, however, is impossible because, by Proposition~\ref{th-fluid-infinite-init}, 
 w.p.1 from any subsequence of $r$ we can find a further 
subsequence such that:
\begin{align*}
& \sum_i \psi^r_{ij}(t) \to \sum_i \psi_{ij}(t)=\beta_j 
~~\mbox{uniformly in $[T_1,T_2]$},~~\forall j,\\
& d^r_{ij}(T_2)-d^r_{ij}(T_1) \to \int_{T_1}^{T_2} \mu_{ij} \psi_{ij}(t)dt, ~~\forall (ij),
\end{align*}
and therefore
\begin{multline*}
\sum_{(ij)} \nu_i (d^r_{ij}(T_2)-d^r_{ij}(T_1)) \to 
\sum_{(ij)} \nu_i \int_{T_1}^{T_2} \mu_{ij} \psi_{ij}(t)dt =\\
\sum_{(ij)} \int_{T_1}^{T_2} \frac{\alpha_j}{\beta_j} \psi_{ij}(t)dt =
\sum_j \frac{\alpha_j}{\beta_j} \int_{T_1}^{T_2} [\sum_i \psi_{ij}(t)]dt = T_2-T_1.
\end{multline*}
\end{proof}

\begin{proof}[Proof of Theorem~\ref{thm-fluid-stability}] 
Recall that $\nu_i>0$ is the workload associated with a single request of type $i$; i.e., the optimal dual variable associated with \eqref{lp-constraint1} for type $i$. We consider the quantity
\[
W^r(t) = \sum_i \nu_i x^r_i(t)
\]
(where $x^r_i(t) = \sum_i q^r_i(t) + \sum_{ij} \psi^r_{ij}(t)$), the total workload of the system. We will argue that the quantity
\[
\CL^r(t) = [W^r(t)]^2
\]
will serve as a Lyapunov function for the $r$th system. Namely, the following property holds:
there exist positive constants $K$, $T$, $C_1$, $C_2$, $C_3$ such that, for all sufficiently large $r$, 
\beql{eq-lyap1}
\text{if }\CL^r(t) > K \text{ then } \BE[\CL^r(t+T) - \CL^r(t) \vert \CL^r(t)] < -C_1 W^r(t) +C_2
\end{equation}
and 
\beql{eq-lyap2}
\text{if }\CL^r(t) \le K \text{ then } \BE[\CL^r(t+T) - \CL^r(t) \vert \CL^r(t)] < C_3.
\end{equation}
(The proof of \eqn{eq-lyap1}-\eqn{eq-lyap2} is given after we complete the theorem proof.)
It is then a standard application of the Foster-Lyapunov criteria to conclude that for all sufficiently large $r$ the system
Markov process is positive recurrent,
and moreover, the stationary distributions are such that $\BE W^r = \sum_i \nu_i \BE x^r_i$ remains uniformly (in $r$) bounded. 
Indeed, for any fixed initial state 
of the process, consider the embedded chain at times $0, T, 2T, \ldots$.
It easily follows (using the fact that each input flow is Poisson,
and fluid scaling is applied) that $0\le \BE\CL^r(nT)=\BE[W^r(nT)]^2 < \infty$ for all 
$n=0,1,2, \ldots$. Also, WLOG, by rechoosing if necessary $C_1$ and $C_2$,
we can assume that the ``then'' part of \eqn{eq-lyap1} holds for any $\CL^r(t)$.
We see that, for any $n$,
$$
\BE\CL^r((n+1)T) -  \BE\CL^r(nT) \le -C_1 \BE W^r(nT) +C_2;
$$
from here the positive recurrence and steady-state bound $\BE W^r\le C_2/C_1$ 
easily follow, because the opposite would imply $\BE\CL^r(nT)\to -\infty$
as $n\to\infty$.

Uniform bound on $\BE W^r$ implies tightness of invariant distributions.
The tightness together with Proposition~\ref{th-fluid-drains} imply that the sequence of invariant distributions
must weakly converge to the point mass at equilibrium.

It remains to show property \eqn{eq-lyap1}-\eqn{eq-lyap2}.
First, it is easy to see (and is a standard observation) that 
\beql{eq-uniform222}
\forall T>0,
~~ \BE[W^r(t+T)-W^r(t)]^2 ~\mbox{are uniformly bounded across all $r$ and $t$}. 
\end{equation}
This guarantees \eqn{eq-lyap2} for any fixed $K$. To prove \eqn{eq-lyap1}, we fix $T_1>0$ as in Lemma~\ref{lem-for-lyapunov}, and then choose a large fixed $T > T_1$. Note that
\[
(\min_{i \in \CI} \nu_i)(\max_{i \in \CI} q^r_i(t)) \leq W^r(t) \leq (\max_{i \in \CI} \nu_i)(I \max_{i \in \CI}q^r_i(t) + \sum_j \beta_j);
\]
in particular,  the condition $\max_{i \in \CI} q^r_i(0) \to\infty$ in Lemma~\ref{lem-for-lyapunov} is equivalent to $W^r(0)\to\infty$.
If we fix a sufficiently small $\epsilon'>0$ and apply Lemma~\ref{lem-for-lyapunov}, we obtain the following fact:\\ for a sufficiently large fixed $K>0$ (as a function of $T$),
uniformly on all $\CL^r(0)>K$ and all large $r$, 
\beql{eq-w-drift}
\BP \{W^r(T) - W^r(0) \leq 2 \rho T_1 -\frac12(1-\rho) (T - T_1)\} \ge 1-\epsilon'.
\end{equation}
Indeed, the $2 \rho  T_1$ is a crude upper bound on $W^r(T_1) - W^r(0)$, which holds with high probability (w.h.p.) for large $r$, since by \eqn{eqn:workload} new workload arrives at average rate $\rho$ (in the fluid-scaled system).
The term $-(1/2)(1-\rho) (T - T_1)$ is an upper bound on $W^r(T) - W^r(T_1)$, 
also holding w.h.p., because
by Lemma~\ref{lem-for-lyapunov} the average rate at which workload leaves the system
is w.h.p. close to $1$ in $[T_1,T]$; and recall that $\rho<1$. This proves
\eqn{eq-w-drift}. The RHS of the first inequality in \eqn{eq-w-drift} is negative 
if we choose $T$ large enough. This, along with \eqn{eq-uniform222}, 
implies \eqn{eq-lyap1}, since we have the identity $\CL^r(t+T)-\CL^r(t) = 2W^r(t) (W^r(t+T)-W^r(t)) + [W^r(t+T)-W^r(t)]^2$.
\end{proof}

\section{Proof of Theorem~\ref{th-stationary-scale-main}(ii)}
\label{sec-main-proof}

\subsection{Preliminaries}

In the previous section we have shown that the process 
$(\Psi^r_{ij}(\cdot), Q^r_i(\cdot))$ is positive recurrent, and then
has unique stationary (or invariant) distribution for all large $r$
(which proved Theorem~\ref{th-stationary-scale-main}(i)).
Moreover, 
\beql{eq-yyy}
r^{-1}(\Psi^r_{ij}-\psi_{ij}^* r, Q^r_i) \implies 0.
\end{equation}
Here and in the rest of the paper
$(\Psi^r_{ij}, Q^r_i)$ means ``$(\Psi^r_{ij}(t), Q^r_i(t))$ in stationary regime.''

So, we know that Theorem~\ref{th-stationary-scale-main}(ii)
is true for $\epsilon=1/2$, and our goal is to prove it for any $\epsilon>0$.
In what follows $0< \epsilon< 1/2$ is fixed.

From \eqn{eq-yyy}, 
for an arbitrarily small fixed $\delta>0$, we can choose a positive function $g(r)=o(r)$, such that,
\beql{eq-sub-r}
\BP\{ \abs{(\Psi^r_{ij} - r \psi^*_{ij}, Q^r_i)} \le g(r)\} \ge 1-\delta.
\end{equation}
Without loss of generality, assume $r^{-1/2-\epsilon}g(r)\to\infty$.

We will prove that there exist positive constants $C$ and $T$, such that for any
fixed $\delta_1>0$ the following holds for all sufficiently large $r$: 
\begin{multline}
\label{eq-goal}
\abs{\Bigl(\Psi^r_{ij}(0) - r \psi^*_{ij}, Q^r_i(0)\Bigr)} \le g(r) ~~
\mbox{implies} \\
\BP\{  \abs{\Bigl(\Psi^r_{ij}(T\log r) - r \psi^*_{ij}, Q^r_i(T\log r)\Bigr)} 
\le C r^{1/2+\epsilon}\}
\ge 1-\delta_1.
\end{multline}
This fact, along with \eqn{eq-sub-r}, 
implies that for all large $r$, in steady-state,
$$
\BP\{  \abs{\Bigl(\Psi^r_{ij} - r \psi^*_{ij}, Q^r_i\Bigr)} 
\le C r^{1/2+\epsilon}\}
\ge (1-\delta)(1-\delta_1).
$$
This clearly proves Theorem~\ref{th-stationary-scale-main}(ii),
because $\delta, \delta_1$ can be chosen, and $\epsilon$ rechosen, to be arbitrarily small.
So, the rest of Section~\ref{sec-main-proof} is the proof of \eqn{eq-goal}, with the final part of the proof given in Section~\ref{subsec-main-proof}.

We will need FSLLN-type results, which can be obtained 
from a strong approximation of Poisson processes, 
available e.g. in \cite[Chapters 1 and 2]{Csorgo_Horvath}: 
\begin{prop}\label{thm:strong approximation-111-clean}
A unit rate Poisson process $\Pi(\cdot)$ and 
a standard Brownian motion $W(\cdot)$ can be constructed on a common
probability space in such a way that the following holds.
For some fixed positive constants $C_1$, $C_2$, $C_3$,   such that 
$\forall T>1$ and
$\forall u \geq 0$
\[
\BP\left(\sup_{0 \leq t \leq T} \abs{\Pi(t) - t - W(t)} \geq C_1 \log T + u\right) \leq C_2 e^{-C_3 u}.
\]
\end{prop}
From here, for the unit rate Poisson processes $\Pi^{(a)}_i(\cdot)$ and 
$\Pi^{(s)}_{ij}(\cdot)$, driving exogenous arrivals and departures,
we obtain the following fact. (For $\Pi^{(a)}_i(\cdot)$, for example,
we replace $t$ with $\lambda_i rt$; $T$ with $\lambda_i rT \log r$;
and $u$ with $r^{1/4}$.)
\begin{prop}\label{thm:strong approximation-111}
For any fixed $T>0$ and 
any subsequence of $r\to\infty$, we can find a further 
subsequence (with $r$ increasing sufficiently fast),
such that:\\
 for each $i$
\[
\sup_{0 \leq t \leq T\log r} r^{-1/2-\epsilon/2} \abs{\Pi^{(a)}_i(\lambda_i r t) - \lambda_i rt} 
\to 0, ~~~\mbox{w.p.1},
\]
and for each $(ij)$
\[
\sup_{0 \leq t \leq T\log r} r^{-1/2-\epsilon/2} \abs{\Pi^{(s)}_{ij}(\mu_{ij}\beta_j r t) - \mu_{ij} \beta_j rt} 
\to 0, ~~~\mbox{w.p.1}.
\]
\end{prop}

Let $F^r(t)$ be the process of (unscaled) deviations from equilibrium; that is,
\[
F^r(t) = (\Psi^r_{ij}(t) - r \psi^*_{ij}, Q^r_i(t)).
\]
Suppose we have a function $h(r)$, such that $r^{1/2+\epsilon} \le h(r) \le g(r)$.
(The quantity $h(r)$ will be the ``scale'' of  $\abs{F^r(0)}$; sometimes, we simply use
$h(r)=\abs{F^r(0)}$, but not necessarily.)
We will establish properties of $F^r(\cdot)$ under two different 
scalings, called hydrodynamic and local-fluid.

We remark that the use of multiple scalings (in addition to the ``standard" fluid scaling)
 is typical in the analysis of systems in many-server asymptotic regime,
cf. \cite{Gurvich_Whitt} and references therein.  However, our hydrodynamic and local-fluid scalings are somewhat unusual
in that the scaling factor $h(r)$ is strictly ``between" $r$ and $r^{1/2}$. (When $h(r)=r$,
both local-fluid and hydrodynamic scalings become the standard fluid scaling; if 
$h(r)=r^{1/2}$,
the local-fluid scaling becomes the standard diffusion scaling.)
Also, the system behavior, of course, depends on the control discipline, LAP in our case; and so our analysis of LAP under various scalings is new.
Most importantly, the way we use these multiple scalings for the purposes of proving tightness of {\em stationary} distributions
is novel, to the best of our knowledge.

\subsection{Hydrodynamic scaling}
\label{subsec-hydro}

Consider the process under the following scaling and centering:
\begin{multline}
(\ov \psi^r_{ij}(t),\ov q^r_i(t), \ov x^r_i(t), \ov a^r_i(t), \ov d^r_{ij}(t),\ov \xi^r_{ij}(t)) = \\
h(r)^{-1} \Bigl(\Psi^r_{ij}((h(r) r^{-1}t)-r \psi^*_{ij},Q^r_i(h(r) r^{-1}t), X^r_i(h(r) r^{-1}t) - r \sum_j \psi^*_{ij}, \\
A^r_i(h(r) r^{-1}t), D^r_{ij}(h(r) r^{-1}t),\Xi^r_{ij}(h(r) r^{-1}t) \Bigr).
\end{multline}

\begin{theorem}\label{th-hydro-convergence}
Consider a sequence of deterministic realizations, such that the driving realizations
satisfy FSLLN conditions, namely:
\beql{eq-lln-hydro1}
(\ov a^r_i(t), ~t\ge 0) \to (\lambda_i t, ~t\ge 0), ~~\mbox{u.o.c.}, ~\forall i
\end{equation}
\beql{eq-lln-hydro2}
\Bigl(h(r)^{-1} \bigl( D^r_{ij}(h(r)r^{-1} t) - \mu_{ij} \int_0^{h(r)r^{-1} t} \Psi^r_{ij}(s) ds \bigr), ~t\ge 0\Bigr) \to 0,
~~\mbox{u.o.c.}, ~\forall (ij).
\end{equation}
Suppose $(\ov \psi^r_{ij}(0),\ov q^r_i(0)) \to (\ov \psi_{ij}(0), \ov q_i(0))$.\\
Then, 
for any subsequence of $r$ there exists a further subsequence along which 
$(\ov \psi^r_{ij}(\cdot),\ov q^r_i(\cdot), \ov x^r_i(\cdot), \ov a^r_i(\cdot), \ov d^r_{ij}(\cdot),\ov \xi^r_{ij}(\cdot))$ converges uniformly on compact sets to a set $(\ov \psi_{ij}(\cdot), \ov q_i(\cdot), \ov x_i(\cdot), \ov a_i(\cdot), \ov d_{ij}(\cdot),\ov \xi_{ij}(\cdot))$ of Lipschitz continuous functions satisfying conditions \eqref{eqn:hydro-model}. 
(The conditions involving derivatives are to be satisfied 
at regular time points $t\ge 0$ of the limiting set of functions.)
\end{theorem}

The hydrodynamic model conditions are: 
\begin{subequations}\label{eqn:hydro-model}
\begin{equation}\label{eqn-first-hydro}
\ov q_i(t) \geq 0, ~~ \forall i \in \CI;  \qquad \sum_i \ov \psi_{ij}(t) \leq 0, ~~ \forall j \in \CJ
\end{equation}
\begin{equation}
\ov a_i(t) = \lambda_i t, ~~ \forall i \in \CI; \qquad \ov d_{ij}(t) = \mu_{ij}\psi^*_{ij} t, ~~ \forall (ij) \in \CE
\end{equation}
\begin{align}\label{eqn-q-hydro}
\ov q_i(t) = \ov q_i(0) +  \ov a_i(t) - \sum_j \ov \xi_{ij}(t),~~ \forall i \in \CI;\\
\ov \psi_{ij}(t) =\ov \psi_{ij}(0) +  \ov \xi_{ij}(t) - \ov d_{ij}(t),~~ \forall i \in \CI \notag
\end{align}
\begin{equation}\label{eqn-x-hydro}
\ov x_i(t) = \ov q_i(t) + \sum_j \ov \psi_{ij}(t) \equiv  \ov x_i(0),~~ \forall i \in \CI
\end{equation}
\begin{equation}\label{eqn-key1-hydro}
\sum_i  \ov \psi_{ij}(t) = 0, ~~ \text{whenever $\ov q_{i'}(t)>0$ for at least one $i' \in \CC(j)$}
\end{equation}
\begin{equation}\label{eqn-key2-hydro}
\frac{d}{dt}\ov \xi_{ij}(t) = 0 , ~~ \text{whenever $\ov q_{i'}(t)>0$ for at least one $i' \in \CC(j)$, $i'<i$}
\end{equation}
\begin{equation}\label{eqn-key3-hydro}
\frac{d}{dt}\ov \xi_{ij}(t) = 0 , ~~ \text{whenever $\sum_{k} \ov \psi_{kj'}(t) < 0$ for at least one $(ij') < (ij)$}
\end{equation}
\begin{multline}\label{eqn:last-hydro-2}
\frac{d}{dt} \ov \xi_{ij}(t) = 
\sum_{i'} \mu_{i'j} \psi^*_{i'j} -  \sum_{(i'j) < (ij)} \frac{d}{dt} \ov \xi_{ij'}(t),\\
\text{whenever $\ov q_i(t)>0$ (and then necessarily $\sum_{k} \ov \psi_{kj}(t) = 0$)}
\end{multline}
\begin{multline}\label{eqn:last-hydro-1}
\frac{d}{dt} \ov \xi_{ij}(t) = \lambda_{i} - \sum_{(ij') < (ij)} \frac{d}{dt} \ov \xi_{ij'}(t), 
\\
\text{whenever $\sum_{k} \ov \psi_{kj}(t) < 0$ (and then necessarily $\ov q_i(t)=0$)}
\end{multline}
\begin{multline}\label{eqn:last-hydro}
\frac{d}{dt} \ov \xi_{ij}(t) = \min \left(\lambda_{i} - \sum_{(ij') < (ij)} \frac{d}{dt} \ov \xi_{ij'}(t), 
\sum_{i'} \mu_{i'j} \psi^*_{i'j} -  \sum_{(i'j) < (ij)} \frac{d}{dt} \ov \xi_{ij'}(t) \right)\\
\text{whenever $\ov q_i(t)=0$ and $\sum_{k} \ov \psi_{kj}(t) = 0$.}
\end{multline}
\end{subequations}

There is a clear correspondence between the hydrodynamic model and fluid model
conditions. This is not surprising, of course, -- the hydrodynamic limit is also an
FSLLN-type limit, but on a different, finer time and space scale.
We omit the proof of Theorem~\ref{th-hydro-convergence} -- it is analogous to that
of Proposition~\ref{th-local-fluid-convergence}.

We call any Lipschitz solution $(\ov \psi_{ij}(\cdot), \ov q_i(\cdot), \ov x_i(\cdot), \ov a_i(\cdot), \ov d_{ij}(\cdot), \ov \xi_{ij}(\cdot))$
of \eqref{eqn:hydro-model}  a \emph{hydrodynamic model} (HM) of the system with initial state $(\ov \psi_{ij}(0), \ov q_i(0))$; a set $(\ov \psi_{ij}(\cdot), \ov q_i(\cdot))$,
which is a projection of an HM we often call a hydrodynamic model as well.
Also, we sometimes use shorter notations 
$\ov f^r(\cdot)=(\ov \psi_{ij}^r(\cdot), \ov q_i^r(\cdot))$,
$\ov f(\cdot)=(\ov \psi_{ij}(\cdot), \ov q_i(\cdot))$.

We have the following corollary of Theorem~\ref{th-hydro-convergence} 
which we record for future reference. 

\begin{coroll}
\label{close-hydro}
For any fixed $T>0$, $K>0$ and $\delta_2>0$, there exists a sufficiently small $\delta_3>0$,
such that the following holds. Uniformly on all $|\ov f^r(0)| \le K$ and all sufficiently
large $r$, conditions 
\beql{eq-lln-hydro111}
\max_i \sup_{[0,T]} |\ov a^r_i(t)-\lambda_i t| \le \delta_3,
\end{equation}
\beql{eq-lln-hydro222}
\max_{(ij)} \sup_{[0,T]} 
|h(r)^{-1} \bigl( D^r_{ij}(h(r)r^{-1} t) - \mu_{ij} \int_0^{h(r)r^{-1} t} \Psi^r_{ij}(s) ds \bigr)| \le \delta_3,
\end{equation}
imply 
\beql{eq-lln-hydro333}
\sup_{[0,T]} |\ov f^r(t)-\ov f(t)| \le \delta_2,
\end{equation}
where $\ov f(\cdot)$ is a hydrodynamic model.
\end{coroll}

\begin{proof} 
Suppose not. Fix $T, K, \delta_2$. There must exist a sequence $\delta_3 \downarrow 0$, 
and a corresponding sequence $r=r(\delta_3)\uparrow \infty$,
such that \eqref{eq-lln-hydro111}, \eqref{eq-lln-hydro222} 
and the convergence $\ov f^r(0)\to \ov f(0)$
of initial states hold, 
but \eqref{eq-lln-hydro333} fails for {\em any} hydrodynamic model.
This, however, is impossible, because according to Theorem~\ref{th-hydro-convergence}
(or rather its version, specialized to a finite time interval, to be precise) we can 
choose a further subsequence of $r$ along which $\ov f^r(t)\to \ov f(t)$,
uniformly in $[0,T]$, where $\ov f(\cdot)$ is a hydrodynamic model
starting from $\ov f(0)$.
\end{proof}

\begin{theorem}\label{hydro-behavior}
For any $K > 0$ there exists a finite time $T = T(K)$ 
and constant $C=C(K)>0$ such that all hydrodynamic models
with $\abs{\ov f(0)} \leq K$ satisfy the following
conditions:
$\sum_i \ov \psi_{ij}(T) = 0, \forall j < J$, $\ov q_i(T) = 0, \forall i \in \CI$,
and $\ov f(t)\equiv \ov f(T)$ for all $t \geq T$;
$\max_{t\ge 0} |\ov f(t)| \le CK$.
Moreover, \linebreak $(\ov \psi_{ij}(T), \ov q_i(T)) = 
L (\ov \psi_{ij}(0), \ov q_i(0))$, where $L$ is a fixed linear mapping.
\end{theorem}

\begin{proof}
Consider a fixed HM $\ov f(\cdot)$. 
Consider the highest priority activity $(1j)$. There are two possible cases:
$j$ is a leaf or $1$ is a leaf.\\
Case a:
If $j$ is a leaf, then $\ov \psi_{1j}(t) \leq 0$ at all times, and $\ov\psi_{1j}(t)$ must increase at positive rate, bounded away from $0$, until it reaches
$0$ within a finite time. Thereafter, $\ov\psi_{1j}(t)$ will stay at 0. 
(The argument is very similar to Case b in the proof of Proposition~\ref{th-fluid-drains}.)\\
Case b:
If type $1$ is a leaf, then $\ov q_1(t)$ must decrease and
$\ov\psi_{1j}(t)$ increase at the same rate (positive, bounded away from $0$), until
the entire queue (if any) ``relocates into'' $\ov\psi_{1j}$; after that time,
$\ov\psi_{1j}(t)$ and $\ov q_1(t)=0$ will not change.\\
We see that in either case a or b, after a finite time, 
the highest priority activity $(1j)$ can be in a sense ``ignored".
This allows us to proceed by induction on the activities, from the highest priority to the lowest,
to check that by some finite time $T$ 
(depending on $K$) the hydrodynamic model gets into a state
$\ov f(T)$, satisfying conditions of the theorem, and will
stay in this state for all $t\ge T$. Since all HMs are uniformly Lipschitz,
we obviously have a uniform bound $\max_{t\ge 0} |\ov f(t)| \le CK$ for some $C$.

Furthermore, since all $\ov x_i(t)$ do not change with time, 
the linear mapping $L$ is as follows: 
$L(u_{ij},w_i) = (c_{ij},0)$ where $(c_{ij})$ is the unique solution to
\begin{subequations}\label{eqn:starting state-hydro}
\begin{equation}
\sum_j u_{ij} + w_i = \sum_j c_{ij},~~ \forall i \in \CI
\end{equation}
\begin{equation}
\sum_i c_{ij} = 0,~~ \forall j < J
\end{equation}
\end{subequations}
\end{proof}

\begin{remark}
Examination of the proof of Theorem~\ref{hydro-behavior}
reveals that the HM for any initial state is in fact unique.
Moreover, with a little further argument, it is easy to show that an HM depends
on the initial state continuously. Furthermore, the HMs are 
scalable: if $(\ov f(t), ~t\ge 0)$ is an HM, then so is $(\ov f(ct)/c, ~t\ge 0)$
for any $c>0$. From here, it is easy to find that
the theorem statement holds for a constant $C$ independent of $K$ and for $T=CK$.
We will not need these stronger properties in this paper.
\end{remark}

For future reference, note that $L(u_{ij},w_i)= (c_{ij},0)$ is a function only of the vector $(z_i)$, where $z_i = w_i + \sum_j u_{ij}$. The corresponding linear mapping from $(z_i)$ to $(c_{ij})$, we denote $L'$.

\subsection{Local-fluid scaling}
\label{subsec-lf}

The process under \emph{local fluid scaling} is as follows. For each $r$ consider
\[
(\tilde \psi^r_{ij}(t), \tilde q^r_i(t)) \equiv \tilde f^r(t) = h(r)^{-1} F^r(t).
\]
We will also denote $\tilde x_i^r(t) =  h(r)^{-1} X_i^r(t) \equiv \tilde q^r_i(t) + \sum_j \tilde\psi^r_{ij}(t)$.

Since $\tilde \psi^r_{ij}(\cdot)$ (as well as $\ov \psi^r_{ij}(\cdot)$) is centered
before it is scaled in space, we in particular have (by Assumption~\ref{ass:rho}) that
 $\sum_i \tilde\psi^r_{ij}(t) \leq 0$ for all $j < J$ at all times $t$. 

\begin{theorem}
\label{thm:local fluid limit-simple}
Consider a sequence of deterministic realizations, such that the driving realizations
satisfy FSLLN conditions, namely:
\beql{eq-lln-lf1}
(h(r)^{-1} ( A^r_i(t) - \lambda_i r t), ~t\ge 0) \to 0, ~~\mbox{u.o.c.}, ~\forall i
\end{equation}
\beql{eq-lln-lf2}
\Bigl(h(r)^{-1} \bigl( D^r_i(t) - \mu_{ij} \int_0^t \Psi^r_{ij}(s) ds \bigr), ~t\ge 0 \Bigr) \to 0, 
 ~~\mbox{u.o.c.}, ~\forall (ij).
\end{equation}
Assume that the initial states converge to a fixed vector
$(\tilde  \psi^r_{ij}(0),\tilde q^r_i(0)) \to (\tilde  \psi_{ij}(0),  \tilde q_i(0))$. 
Further assume that $\tilde q_i(0)=0, \forall i,$ and 
$\sum_{i} \tilde \psi_{ij}(0)=0$ for all $j<J$. 
(In other words, 
$(\tilde  \psi_{ij}(0),  \tilde q_i(0))=L(\tilde  \psi_{ij}(0),  \tilde q_i(0))$.)
Then, 
for any subsequence of $r$ there exists a further subsequence along which  
\beql{eq-new2}
(\tilde \psi^r_{ij}(\cdot), \tilde q^r_i(\cdot)) \to (\tilde\psi_{ij}(\cdot), \tilde q_i(\cdot)),
~~\mbox{u.o.c.},
\end{equation}
where $(\tilde\psi_{ij}(\cdot), \tilde q_i(\cdot)$ is a set of Lipschitz functions,
with initial condition \linebreak $(\tilde \psi_{ij}(0), \tilde q_i(0))$,
satisfying (local fluid model) conditions \eqref{eqn:local fluid model}.
Moreover, 
these limit trajectories  depend continuously on the initial state and are such that, 
uniformly on all of them,
\beql{eq-new1}
|(\tilde \psi_{ij}(t), \tilde q_i(t))| \le 
|(\tilde \psi_{ij}(0), \tilde q_i(0))| c_1 e^{-c_2 t}, ~~\forall t\ge 0,
\end{equation}
where $c_1, c_2 >0$ are fixed constants.
\end{theorem}

The local fluid model conditions are as follows:
\begin{subequations}\label{eqn:local fluid model}
\begin{equation}
\tilde q_i(t) = 0, \quad \forall i \in \CI
\end{equation}
\begin{equation}
\sum_j \tilde \psi_{ij}(t) = \sum_j \tilde \psi_{ij}(0) - \sum_j\int_0^t \mu_{ij} \tilde \psi_{ij}(s) ds, \quad \forall i \in \CI
\end{equation}
\begin{equation}
\sum_i \tilde \psi_{ij}(t) = 0, \quad \forall j < J
\end{equation}
\end{subequations}
The $I+J-1$ equations for the $I+J-1$ functions $(\tilde \psi_{ij}(\cdot))$ 
can be solved sequentially, in order of decreasing activity priority, since the highest unsolved-for priority will always correspond to either a customer-type or a server-type leaf of the remaining activity tree. Any Lipschitz trajectory satisfying 
\eqn{eqn:local fluid model} 
we will call a {\em local fluid model} (LFM). Conditions \eqn{eqn:local fluid model} reduce to a system of linear ODEs
for $(\tilde \psi_{ij}(t))$,  which of course implies the continuous dependence
on initial state; the fact that each LFM converges to $0$
is easily established, again by induction on activities; therefore, we obtain the uniform
exponential bound
\eqn{eq-new1}.

 Analogously to $\tilde f^r(\cdot)=(\tilde \psi^r_{ij}(\cdot), \tilde q^r_i(\cdot))$,
we will use shorter notation $\tilde f(\cdot)=(\tilde \psi_{ij}(\cdot), \tilde q_i(\cdot))$.

\begin{proof}[Proof of Theorem~\ref{thm:local fluid limit-simple}]
The non-trivial part of the proof is  
showing the Lipschitz property 
of the limit $\tilde f(\cdot)$,
because it is no longer a simple consequence 
of the FSLLN for the driving processes (as it was for the fluid and hydrodynamic
limits). This is because the arrival and service rates in the system (with index $r$)
are $O(r)$, while the space is scaled down by $h(r)=o(r)$. For the same reason,
it is also not ``automatic'' that the limit queues $\tilde q_i(\cdot)$ stay
at $0$. This difficulty is resolved as follows. 
Consider arbitrary number 
$C_4 > \|(\tilde  \psi_{ij}(0))\|$, and the random time 
$\tau(r)=\min\{t~|~\|(\tilde  \psi_{ij}^r(t))\|\ge C_4\}$. Then, speaking informally 
(the formal statements are given below), the trajectory $\tilde  x_{i}^r(\cdot)$ for each $i$
must be ``almost Lipschitz'' in the interval
$[0,\tau(r)]$, with the Lipschitz constant $\eta = C_4 \|(\mu_{ij})\|$,
because the absolute
difference between
the arrival and departure rates (scaled down by $h(r)$) is upper 
bounded by $\eta$ in $[0,\tau(r)]$;  similarly, 
each queue
length trajectory $\tilde  q_{i}^r(\cdot)$ is ``almost Lipschitz'' in $[0,\tau(r)]$.

Formally, it is easy to show the following:
if $\tau(r)\to 0$ along some subsequence,
then (denoting $\tilde x_i(0)=\sum_j \tilde \psi_{ij}(0)$)
\beql{cond11}
\sup_{[0,\tau(r)]} \|(\tilde  x_{i}^r(t)) - (\tilde  x_{i}(0))\| \to 0,
~~~~\sup_{[0,\tau(r)]} \|(\tilde  q_{i}^r(t)) - (\tilde  q_{i}(0))\| \to 0.
\end{equation}
If $\liminf  \tau(r) > \epsilon_4 >0$ along some subsequence, then there exists
a further subsequence along which 
\beql{cond22}
(\tilde  x_{i}^r(\cdot))\to (\tilde  x_{i}(\cdot)),~~~~
(\tilde  q_{i}^r(\cdot))\to (\tilde  q_{i}(\cdot)),
\end{equation}
where the convergences are uniform in $[0,\epsilon_4]$, and
each function $\tilde  x_{i}(\cdot)$ and $\tilde  q_{i}(\cdot)$
is Lipschitz with constant $\eta$
in $[0,\epsilon_4]$.

In the case $\tau(r)\to 0$, as a consequence of \eqn{cond11}, we also must have
\beql{eq-contra1}
\sup_{[0,\tau(r)]} \|(\tilde  \psi_{ij}^r(t)) - (\tilde  \psi_{ij}(0))\| \to 0.
\end{equation}
Indeed, if \eqn{eq-contra1} does not hold, then for some fixed $\delta>0$,
we can choose a subsequence of $r$ and a corresponding sequence of
times $\tau_1(r)\in [0,\tau(r)]$ such that 
$\|(\tilde  \psi_{ij}^r(t)) - (\tilde  \psi_{ij}(0))\|\ge \delta$
for the first time.
Fix $T>0$ (we will specify the choice later), 
and consider the sequence of times $\tau_1(r)-T h(r)/r$.
Suppose first that $\tau_1(r)-T h(r)/r \ge 0$ for infinitely many $r$;
then, we consider a further subsequence along which this holds,
and the trajectory on the time interval $[\tau_1(r)-T h(r)/r, \tau_1(r)]$.
Now, if we reset the time origin to $\tau_1(r)-T h(r)/r$ and ``stretch''
the interval $[\tau_1(r)-T h(r)/r, \tau_1(r)]$ by the factor $r/h(r)$,
we will obtain hydrodynamic-scaled trajectories in the interval $[0,T]$.
We then choose a further subsequence of $r$ along which these trajectories 
converge (u.o.c.) to an HM. This HM $\ov f(\cdot)$ will be such that
$(\ov x_i(0))= (\tilde  x_{i}(0))$, 
$\|(\ov \psi_{ij}(0)) - (\tilde  \psi_{ij}(0))\|\le \delta$,
all $\ov q_i(0)=0$, and $\|(\ov \psi_{ij}(T)) - (\tilde  \psi_{ij}(0))\|\ge \delta$.
 We now specify the choice of $T$: it is large enough
so that $(\ov \psi_{ij}(T))=L'(\ov x_{i}(0))$.
But, $(\ov x_{i}(0))=(\tilde x_{i}(0))$, which means
$(\ov \psi_{ij}(T))=L'(\tilde x_{i}(0))=(\tilde  \psi_{ij}(0))$ -- a contradiction.
The contradiction in the case when $\tau_1(r)-T h(r)/r < 0$ for all large $r$
is obtained similarly, except for all $r$ we use time interval
$[0,T h(r)/r]$ to construct a contradicting HM.
Thus, we proved \eqn{eq-contra1}. 
But, this leads to a contradiction 
with the definition of $\tau(r)$. We conclude that the case $\tau(r)\to 0$
is in fact impossible, and we always have the case $\liminf  \tau(r) > \epsilon_4 >0$ and 
\eqn{cond22}. 

Next, in addition to \eqn{cond22}, we show that 
\beql{cond22add}
|\tilde f^r(t) - L \tilde f^r(t)| \to 0,~~~
\mbox{in particular}~|(\tilde \psi_{ij}^r(t)) - L'(\tilde x_{i}^r(t))| \to 0,
\end{equation}
uniformly in $[0,\epsilon_4]$.
(This is, again, proved by contradiction.
If \eqn{cond22add} would not hold, we would be able to construct 
an HM violating the claim of Theorem~\ref{hydro-behavior} 
that $\ov f(t) = L \ov f(t)$
must hold after a finite time. We omit details which are analogous to those
in the proof of \eqn{eq-contra1} above.)

In $[0,\epsilon_4]$ we also have
$$
\tilde x_i(t) = \tilde x_i(0) - \tilde d_i(t), ~~~\forall i,
$$
where the Lipschitz function $\tilde d_i(\cdot)$ is a limit (along a subsequence) of \linebreak
$\sum_j \int_0^t \mu_{ij} \tilde \psi_{ij}^r(s) ds$.

The above properties lead to conditions \eqn{eqn:local fluid model} on the interval
$[0,\epsilon_4]$. \linebreak Namely, we formally define $(\tilde \psi_{ij}(\cdot))=L'(\tilde x_i(\cdot))$,
obtain the convergence $(\tilde \psi_{ij}^r(\cdot)) \to (\tilde \psi_{ij}(\cdot))$ from
\eqn{cond22add}, and then \eqn{eqn:local fluid model} follows.

Finally, as already observed earlier,
the linear ODE \eqn{eqn:local fluid model} solutions satisfy condition
\eqn{eq-new1}. In particular, each local fluid model remains bounded
in $[0,\infty)$. This in turn allows us to conclude that by choosing a sufficiently
large $C_4$, the corresponding $\epsilon_4$ can be arbitrarily large.
This completes the proof.
\end{proof}

We will actually need a generalized version of Theorem~\ref{thm:local fluid limit-simple}.

\begin{theorem}
\label{thm:local fluid limit-general}
Consider a sequence of deterministic realizations, such that the driving realizations
satisfy \eqn{eq-lln-lf1}-\eqn{eq-lln-lf2}.
Assume that the initial states converge to a fixed vector
$\tilde f^r(0) \to \tilde f^{\circ}(0)$.
(We do {\em not} assume $\tilde f^{\circ}(0)=L\tilde f^{\circ}(0)$.)
Then, 
for any subsequence of $r$ there exists a further subsequence along which  
\beql{eq-new2555}
\tilde f^r(\cdot) \to \tilde f(\cdot),
\end{equation}
uniformly on compact subsets of $[0,\infty)$ not containing $0$,
where $\tilde f(\cdot)$ is a local fluid model
with initial state $\tilde f(0)= L\tilde f^{\circ}(0)$.
(Recall that \eqn{eq-new1} holds for any LFM.) 
In addition, for any $K>0$ there exists $C=C(K)>0$ such that
$|\tilde f^{\circ}(0)|\le K$ implies
\beql{eq-jump-bound}
\limsup_{r\to\infty} \sup_{0\le t \le 1} 
|\tilde f^r(t)|\le CK.
\end{equation}
\end{theorem}

\begin{proof}
The proof is a slight generalization of that of Theorem~\ref{thm:local fluid limit-simple}.
For a fixed $T_5>0$ consider the interval $[0,T_5 h(r)/r]$, and the corresponding
hydrodynamic-scaled trajectories in the interval $[0,T_5]$;
$T_5$ is chosen large enough so that the hydrodynamic model
reaches state $\tilde f(0) = L\tilde f^{\circ}(0)$ by time $T_5$.
Then, we must have $\tilde f^r(T_5 h(r)/r)
\to \tilde f(0)$; moreover, by Theorem~\ref{hydro-behavior}
and Corollary~\ref{close-hydro}, 
$$
\limsup_{r\to\infty} \sup_{0\le t \le T_5 h(r)/r} 
|\tilde  f^r(t)|\le CK,
$$
for some $C$, when $|\tilde  f^{\circ}(0)|\le K$.

Then, we consider local fluid scaled trajectories starting time point \linebreak
$T_5 h(r)/r$ (as opposed to $0$), and 
the rest of the proof is essentially same as that of 
Theorem~\ref{thm:local fluid limit-simple}.
\end{proof}

 The following corollary is derived from Theorem~\ref{thm:local fluid limit-general}
analogously to the way Corollary~\ref{close-hydro} was derived from
Theorem~\ref{th-hydro-convergence}.

\begin{coroll}
\label{close-lf}
 For any fixed $K>0$, there exists $C=C(K)>0$ such that the following holds.
For any fixed $T>0$, $\delta_2>0$ and  $\epsilon_2>0$,
there exists a sufficiently small $\delta_3>0$
such that: uniformly on all $|\tilde f^r(0)| \le K$ and all sufficiently
large $r$, conditions 
\beql{eq-lln-lf111}
\max_i \sup_{[0,T]} |h(r)^{-1} ( A^r_i(t) - \lambda_i r t)| \le \delta_3,
\end{equation}
\beql{eq-lln-lf222}
\max_{(ij)} \sup_{[0,T]} 
|h(r)^{-1} \bigl( D^r_i(t) - \mu_{ij} \int_o^t \Psi^r_{ij}(s) ds \bigr)| \le \delta_3,
\end{equation}
imply 
\beql{eq-lln-lf444}
\sup_{[0,T]} |\tilde f^r(t)| \le (K+1)C,
\end{equation}
\beql{eq-lln-lf333}
\sup_{[\epsilon_2,T]} |\tilde f^r(t)-\tilde f(t)| \le \delta_2,
\end{equation}
where $\tilde  f(\cdot)$ is a local fluid model with 
 $|\tilde  f(0) - L \tilde f^r(0)| \le \delta_2$.
\end{coroll}

\subsection{Proof of Theorem~\ref{th-stationary-scale-main}(ii)}
\label{subsec-main-proof}

We are now in position to prove \eqn{eq-goal}, and then 
Theorem~\ref{th-stationary-scale-main}(ii). 
The basic idea is to consider the process in the interval $[0,T\log r]$,
subdivided into $\log r$ intervals, each being $T$-long.
(To be precise, we need to consider an integer number, say $\lfloor \log r \rfloor$,
of subintervals. This does not cause any difficulties besides making notation
cumbersome.)
Then, using the local fluid limit results, we show that, with high probability,
in each of the $T$-long subintervals, the norm $|F^r(t)|$ decreases by a factor
$\delta_6\in (0,1)$, unless the norm $|F^r(t)|$ at the beginning of the subinterval
is below $r^{1/2+\epsilon}$ -- in this case $|F^r(t)|$ will be bounded above by
$3C r^{1/2+\epsilon}$ in the entire subinterval (where $C$ is as in Corollary~\ref{close-lf}
 with $K=2$).
Now, if $\delta_6$ is small enough, so that
\beql{eq-del6}
\delta_6^{\log r} < r^{1/2+\epsilon}/r,~~~\mbox{that is}~~~\delta_6 < e^{-1/2+\epsilon},
\end{equation}
this means that $|F^r(t)|$ must ``dip''
below $r^{1/2+\epsilon}$ at least once, and therefore $|F^r(T\log r)|\le 3C r^{1/2+\epsilon}$
(with high probability). We proceed with the details.

Let us choose $\delta_6>0$ satisfying \eqn{eq-del6}, and then $\delta_2>0$ such that
$2\delta_2 < \delta_6$.
Denote by $|L|$
the norm of the linear operator $L$ (defined in Theorem~\ref{hydro-behavior}),
i.e. the maximum of absolute values of its eigenvalues.
Let us choose $T>0$ large enough so that 
(see Theorem~\ref{thm:local fluid limit-general})
$|L| c_1 e^{-c_2 T} < \delta_2$.

Suppose, for each $r$
the initial state is as in \eqn{eq-goal}. 
To prove \eqn{eq-goal}
it suffices to show that from any 
subsequence of $r$ we can find a further subsequence, along which 
\eqn{eq-goal} holds. So, consider any fixed subsequence, 
and a fixed $\delta_1>0$. 

In each of the subintervals $[(i-1)T,iT]$, $i=1,2,\ldots,\log r$, 
we consider the process with the time origin reset to $(i-1)T$ and the
corresponding initial state $F^r((i-1)T)$; and if $|F^r((i-1)T)|\le g(r)$,
then we set $h(r)=\max(|F^r((i-1)T)|,r^{1/2+\epsilon})$.
(If $|F^r((i-1)T)|> g(r)$ we set $h(r)=g(r)$ for completeness.)
By Proposition~\ref{thm:strong approximation-111},
we can choose a further
subsequence, with $r$ increasing sufficiently fast, so that, w.p.1,
conditions \eqn{eq-lln-lf111} and \eqn{eq-lln-lf222}
hold for all large $r$,
{\em simultaneously}
on each of the subintervals $[0,T]$, $[T,2T]$, ..., $[T(\log r -1),T\log r]$.
 We consider the corresponding local fluid scaled processes $\tilde f^r(\cdot)$,
with their corresponding $h(r)$,
on each of the subintervals; and apply Corollary~\ref{close-lf}.
We see that, with probability 1, for all large $r$, the following
holds for each interval $[(i-1)T,iT]$, $i=1,2,\ldots,\log r$:\\
if $|F^r((i-1)T)| \in [r^{1/2+\epsilon}, g(r)]$ then $|F^r(iT)| \le 
2\delta_2 |F^r((i-1)T)|$;\\
if $|F^r((i-1)T)| < r^{1/2+\epsilon}$ then $|F^r(iT)| \le 
3C r^{1/2+\epsilon}$. \\
Since $2\delta_2 < \delta_6$
we must have $|F^r(iT)| < r^{1/2+\epsilon}$ for at least one $i$.
Finally, we conclude that condition
$|F^r(T\log r)| \le  3C r^{1/2+\epsilon}$ must hold (w.p.1 for all 
large $r$). This obviously implies \eqn{eq-goal}.

\section{Acknowledgments}
The authors would like to thank the referees for their helpful comments.

\bibliographystyle{acmtrans-ims}
\bibliography{tightness}

\end{document}